%% file: main.tex
\documentclass[12pt,reqno]{amsart}
\usepackage[top=1in, bottom=1in, left=1in, right=1in]{geometry}
\usepackage{tikz-cd}
\usepackage{tikz}

\RequirePackage[dvipsnames]{xcolor}

\input{preamble}

\title{Multigraded Regularity of the Complete Flag Variety}
\author{Caitlin M. Davis}
\date{\today}

\begin{document}

\begin{abstract}
We study the multigraded regularity of the complete flag variety under the Pl\"ucker embedding.  In particular, we prove inductive relationships about the regularity regions, and we provide some inner and outer bounds on the regions.
\end{abstract}

\maketitle

\section{Introduction}

Maclagan and Smith introduced in \cite{MaclaganSmith2004} a notion of \textit{multigraded regularity} that generalizes Castelnuovo-Mumford regularity to more general settings, such as products of projective spaces.  In the years since, significant progress has been made in the study of multigraded regularity: \cite{MaclaganSmith2005,SidmanVanTuylWang2006,BerkeschErmanSmith2020,BotbolChardin2017,ChardinNemati2020,BruceCrantonHellerSayrafi2021,BruceCrantonHellerSayrafi2022,CrantonHellerNemati2022,BenderBuseChecaTsigaridas2024}. 
However, major questions remain open.  Notably, to date there are relatively few families of examples for which we have been able to compute the multigraded regularity.  For instance, the multigraded regularity of complete intersections in products of projective spaces was not computed until 2021 (\cite{BruceCrantonHellerSayrafi2021}).

Recent years have also seen many advances in combinatorial techniques for computing Castelnuovo-Mumford regularity, beginning with the regularity formula for Grassmannian matrix Schubert varieties in \cite{RajchgotRenRobichauxStDizierWeigandt2021}.  Subsequent work has identified combinatorial formulas for regularity of ladder determinantal varieties (\cite{RajchgotRobichauxWeigandt2023}), certain Kazhdan-Lusztig varieties (\cite{Robichaux2023}), matrix Schubert varieties (\cite{PechenikSpeyerWeigandt2024}), and tangent cones of Schubert varieties (\cite{Yong2024}).
These formulas witness the fact that regularity is encoded in combinatorial data for these cases of interest.  For cases which lie at the intersection of combinatorial and multigraded commutative algebra, it is natural then to ask whether combinatorial data encode the finer invariant of multigraded regularity.

The complete flag varieties are a well-understood class of varieties which have both rich combinatorial structure and a natural multigrading.  In particular, a complete flag variety can be embedded in a product of projective spaces by the \textit{Pl\"ucker embedding}, which endows the total coordinate ring with a multigraded structure.
The combinatorics of these varieties (and of their Schubert subvarieties) have been leveraged to understand many aspects of their algebro-geometric structure.  (For detailed surveys, see \cite{BilleyGaoPawlowski2025, Brion2005}.)  For instance, their defining ideals (in the Pl\"ucker embedding) have a combinatorial description (see \cite[\S 8.1-9.1]{Fulton1997}).  Furthermore, complete flag varieties decompose into \textit{Schubert cells} which are indexed by elements of $S_n$ and whose dimensions can be computed combinatorially (see \cite[\S 3.3]{BilleyGaoPawlowski2025}).  The \textit{Schubert subvarieties} (closures of the Schubert cells) have a combinatorial description involving Bruhat order on $S_n$ (\cite{Ehresmann1934}, see also \cite[\S 3.4]{BilleyGaoPawlowski2025}).  Although these subvarieties are not smooth in general, there is a remarkably elegant combinatorial criterion for smoothness due to Lakshmibai and Sandhya (\cite{LakshmibaiSadhya1990}, see also \cite{AbeBilley2016}): $\overline{C_{\sigma}}$ is smooth if and only if $\sigma$ avoids $3412$ and $4231$.  Finally, of most relevance in this article is the celebrated theorem of Borel-Weil-Bott (\cite{SerreBorelWeil1995,Bott1957}, see also \cite[\S 4.1]{Weyman2003}), which gives a complete combinatorial description of the cohomology of line bundles on complete flag varieties (Theorem \ref{thm:bottAlg}).

Given both the combinatorial structure of complete flag varieties and our desire to compute multigraded regularity for new families of examples, we aim to address the following question:

\begin{question}\label{q:intro}
    What is the multigraded regularity of the complete flag variety on $\mathbb C^n$ under the Pl\"ucker embedding?
\end{question}

In this section, we will use the notation $\reg(\calF l(\mathbb C^n))$ to denote $\reg(S/I)$, where $I$ is the defining ideal of $\calF l(\mathbb C^n)$ in the Pl\"ucker embedding.  In Section \ref{subsec:regularityDefns}, we will see that it is equivalent, in our setting, to compute the regularity of the coherent sheaf associated to this graded module.  That is, $\reg(S/I) = \reg(\iota_*\calO_{\calF l(\mathbb C^n)})$, where $\iota$ is the Pl\"ucker embedding.

By leveraging the theorem of Borel-Weil-Bott, we are able to fully answer this question for a few low values of $n$.

\begin{example}\label{ex:n=4Intro}
    The multigraded regularity of $\calF l(\mathbb C^4)$ has three minimal elements: $(1,0,0)$, $(0,1,0)$, and $(0,0,1)$ (Proposition \ref{prop:n4}).  To verify that one of these lies in the regularity region requires checking 83 cohomology vanishing conditions.  In particular, $\mathbf d$ is an element of $\reg(\calF l(\mathbb C^4))$ if and only if:
    \[
H^i(\mathbb P^3 \times \mathbb P^5 \times \mathbb P^3, \iota_*\calO_{\calF l(\mathbb  C^4)}(\mathbf d - \mathbf c)) = 0
\]
    for all $0<i\leq \dim(\calF l(\mathbb C^4)) = 6$ and all $\mathbf c \in \mathbb N^3$ with $|\mathbf c| = i$.  When $i = 1$, there are $3$ such choices of $\mathbf c$.  When $i = 2$, there are $6$ such choices of $\mathbf c$.  In total, there are $3 + 6 + 10 + 15 + 21  + 28 = 83$ conditions to verify for each minimal element.
\end{example}

\begin{introtheorem}
We provide a complete description of the regularity region for $n \leq 6$.
\begin{enumerate}
    \item $\reg(\calF l(\mathbb C^3))$ has minimal element $(0,0)$ (Proposition \ref{prop:n=3}).
    \item $\reg(\calF l(\mathbb C^4))$ has minimal elements $(1,0,0)$, $(0,1,0)$, and $(0,0,1)$ (Proposition \ref{prop:n4}).
    \item $\reg(\calF l(\mathbb C^5))$ has 19 minimal elements (Proposition \ref{prop:n=5}).
    \item $\reg(\calF l(\mathbb C^6))$ has 179 minimal elements (Proposition \ref{prop:n=6}, Appendix \ref{app:n=6}).
\end{enumerate}
\end{introtheorem}

As $n$ grows, not only does the number of minimal elements rapidly increase, so does the number of cohomology conditions which must be checked for each element: $9$ conditions when $n=3$, $83$ conditions when $n=4$, $1000$ conditions when $n=5$, $15503$ conditions when $n=6$, and so on.  Crucially, we prove the following inductive results.

\begin{introtheorem}\label{introthm:inductive} Let $\mathbf a \in \operatorname{Pic}(\calF l(\mathbb C^n)) \cong \mathbb Z^{n-1}$.
\begin{enumerate}
    \item Suppose $\mathbf a$ is in $\reg(\calF l(\mathbb C^n))$.  Then $(\mathbf a,b)$ and $(b,\mathbf a)$ are in  $\reg(\calF l(\mathbb C^{n+1}))$, provided \newline $b \geq \dim(\calF l(\mathbb C^{n+1}))$ (Proposition \ref{prop:inductive2}).
    \item Suppose $\mathbf a$ is not in $\reg(\calF l(\mathbb C^n))$.  Then for all $b \in \mathbb Z$, both $(\mathbf a,b)$ and $(b,\mathbf a)$ lie outside $\reg(\calF l(\mathbb C^{n+1}))$ (Proposition \ref{prop:inductive}).
\end{enumerate}
\end{introtheorem}

Once we have computed the regularity for a given value of $n$, this theorem allows us to obtain inner and outer bounds on the regularity region of $\calF l(\mathbb C^{n+1})$.

\begin{example}[Inner bound]\label{ex:introInductiveInner}
Since $\dim(\calF l(\mathbb C^{5})) = 10$,  Example \ref{ex:n=4Intro} implies that
$\reg(\calF l(\mathbb C^{5}))$ contains $(10,1,0,0)+\mathbb N^4$, $(1,0,0,10)+\mathbb N^4$, $(10,0,1,0)+\mathbb N^4$, $(0,1,0,10)+\mathbb N^4$, $(10,0,0,1)+\mathbb N^4$, and $(0,0,1,10)+\mathbb N^4$.
\end{example}

\begin{example}[Outer bound]\label{ex:introInductiveOuter}
Since $(0,0,0)$ is not in $\reg(\calF l(\mathbb C^{4}))$ by Example \ref{ex:n=4Intro}, we can conclude, to take a somewhat arbitrary example, that $(100,28,160,0,0,0,85,1000,5,16)$ is not in $\reg(\calF l(\mathbb C^{11}))$.
\end{example}

Iteratively applying Theorem $\ref{introthm:inductive}$ allows us to study the asymptotic behavior of the multigraded regularity regions.  In particular, we obtain (in Propositions \ref{prop:m(n)lower} and \ref{prop:m(n)upper}) bounds on the minimal total degree of elements of the regularity:

\begin{introtheorem}\label{introthm:totalDeg}
    Let $m(n) = \min\{|\mathbf d|:\mathbf d \in \reg(\calF l(\mathbb C^{n}))\}$ where $|\mathbf d| = d_1+\cdots+d_{n-1}$. For $n \gg 0$, we have:
    \[
    \frac{6}{5}n - 6 \leq m(n) \leq \frac{n^3-n-24}{6}.
    \]
\end{introtheorem}

So the growth of $m(n)$ is at least linear and at most cubic in $n$.  We conjecture that in fact the growth of $m(n)$ is quadratic (Conjecture \ref{conj:m(n)}).

This paper is organized as follows: In \S \ref{sec:background}, we recall relevant definitions and theorems.  In \S \ref{sec:inductive}, we prove some key inductive results. In \S \ref{sec:CMreg}, we study the standard graded regularity of complete flag varieties.  In \S \ref{sec:lowValuesN}, we describe the multigraded regularity for low values of $n$.  In \S \ref{sec:totalDegree}, we leverage our inductive results to study two different metrics on the multigraded regularity of complete flag varieties.

\subsection*{Acknowledgments}
The author is grateful to Daniel Erman for his guidance, support, and feedback.  They would also like to thank Christine Berkesch, Juliette Bruce, Judy Chiang, Federico Galetto, Lauren Cranton Heller, Feiyang Lin, Claudiu Raicu, Jose Israel Rodriguez,  Aleksandra Sobieska, and Anna Weigandt for their valuable insights and suggestions.  Lastly, computations were performed with the aid of Macaulay2 \cite{M2}.


\section{Background}\label{sec:background}

\subsection{Complete flag varieties and the Pl\"ucker embedding}

\begin{definition}
A \textit{complete flag on $\mathbb C^n$} is a chain of subspaces:
\[
V_1 \subset V_2 \subset \cdots \subset V_n = \mathbb C^n
\]
where $\dim V_i = i$ for each $i$.  The \textit{complete flag variety} $\calF = \calF l(\mathbb C^n)$ is the set of all complete flags in $\mathbb C^n$.
\end{definition}

Any invertible $n \times n$ matrix gives a complete flag on $\mathbb C^n$, where $V_i$ is the rowspan of the first $i$ rows.  Furthermore, we can check whether two such matrices define the same flag using the following criterion:

\begin{lemma}
    Let $C$ and $D$ be $j\times n$ matrices (where $j\leq n$).  Then $C$ and $D$ have the same rowspan if and only if they have the same $j \times j$ minors (up to scalar multiple).
\end{lemma}

\begin{example}
\label{ex:plucker}
    Consider the flags on $\mathbb C^3$ defined by the following two matrices:
    \[
    C = \begin{bmatrix}
        4 & 4 & 2\\
        9 & 6 & 3\\
        0 & 4 & 5
    \end{bmatrix}\text{ and } D = \begin{bmatrix}
        2 & 2 & 1\\
        1 & 0 & 0\\
        0 & 1 & 0
    \end{bmatrix}.
    \]
    The $1\times 1$ minors of the first row of $C$ are $4$, $4$, and $2$, and those for $D$ are $2$, $2$, and $1$.  These differ by a multiple of $2$, so both flags have the same one-dimensional subspace.  The $2 \times 2$ minors of the first $2$ rows of $C$ are $-12$, $-6$ and $0$, and those for $D$ are $-2$, $-1$, and $0$.  These differ by a multiple of $6$, so both flags have the same two-dimensional subspace.  Therefore, $C$ and $D$ define the same complete flag.
\end{example}

This is suggestive of an important algebraic structure on complete flag varieties.  Namely, we can embed $\calF l(\mathbb C^n)$ into a product of projective spaces:
\begin{equation*}
    \iota: \calF l(\mathbb C^n) \hookrightarrow \mathbb P^{\binom{n}{1}-1} \times \mathbb P^{\binom{n}{2}-1}\times \cdots \times \mathbb P^{\binom{n}{n-1}-1} = \mathbb P\left(\bigwedge \mathbb C^n\right)
\end{equation*}
If $M$ is an invertible matrix representing a complete flag, then $\iota(M) = M_1 \times \cdots \times M_{n-1}$ where the coordinates of $M_i$ are the $i \times i$ minors of the first $i$ rows of $M$.  This map is called the \textit{Pl\"ucker embedding}. 

\begin{example}
    Consider the flag from Example \ref{ex:plucker}.  The image of this flag under the Pl\"ucker embedding is the point $[4:4:2]\times[-12:-6:0]$, or equivalently $[2:2:1] \times [-2:-1:0]$, in $\mathbb P^2 \times \mathbb P^2$.
\end{example}

We will write $\mathbb P = \mathbb P \left(\bigwedge \mathbb C^n\right)$.  
The inclusion $\iota$ induces $\iota_*$ as follows:
\begin{align*}
    \iota_*\calO_{\calF}(a_1,\ldots,a_{n-1},0) = \calO_{\mathbb P}(a_1-a_2,a_2-a_3,\ldots,a_{n-2}-a_{n-1},a_{n-1}).
\end{align*}

\subsection{Borel-Weil-Bott theorem}

For the original statements and proofs of the theorem of Borel-Weil-Bott, see \cite{SerreBorelWeil1995,Bott1957}.  See also \cite[\S 4.1]{Weyman2003}.  We include below the most relevant statements for our purposes; the original statement is much more general.

\begin{definition}
    Consider $\sigma = (i,i+1) \in S_n$ and $\mathbf a = (a_1,\ldots,a_n) \in \mathbb Z^n$.  The adjacent transposition $\sigma$ acts on $\mathbf a$ by:
    \begin{equation}\label{eqn:swap}
        \sigma\cdot\mathbf a = (a_1,\ldots,a_{i-1},a_{i+1}-1,a_i+1,a_{i+2},\ldots,a_n).
    \end{equation}
\end{definition}

\begin{theorem}[Bott's algorithm]\label{thm:bottAlg}
    If $\mathbf a$ is non-increasing, then:
    \begin{align*}
    &H^0(\calF,\calO_{\calF}(\mathbf a)) \neq 0, \text{ and } \\
    &H^i(\calF,\calO_{\calF}(\mathbf a)) = 0 \text{ for } i\neq 0.
    \end{align*}
    Otherwise, we apply transpositions as in (\ref{eqn:swap}) until one of two possibilities occurs:
    \begin{enumerate}
        \item After some number of transpositions, we obtain $\mathbf a' \in \mathbb Z^n$ where $a'_j = a'_{j+1}-1$ for some $j$.  In this case, $H^i(\calF,\calO_{\calF}(\mathbf a)) = 0 $ for all $i$.
        \item After $N$ transpositions, we obtain a non-increasing sequence.  In this case, 
            \begin{align*}
    &H^N(\calF,\calO_{\calF}(\mathbf a)) \neq 0, \text{ and } \\
    &H^i(\calF,\calO_{\calF}(\mathbf a)) = 0 \text{ for } i\neq N.
    \end{align*}
    \end{enumerate}
\end{theorem}

The following alternative statement will also be useful.

\begin{theorem}[Bott's theorem]\label{thm:bottAlternative}
    Consider $\mathbf b \coloneqq \mathbf a - (0,1,2,\ldots,n-1)$.  If two entries of $\mathbf b$ are equal, then $H^i(\calF,\calO_{\calF}(\mathbf a)) = 0 $ for all $i$.  Otherwise, let $N$ be the number of inversions in $\mathbf b$.  That is, $N$ is the number of pairs $i<j$ such that $b_i<b_j$.  Then
    \begin{align*}
    &H^N(\calF,\calO_{\calF}(\mathbf a)) \neq 0, \text{ and } \\
    &H^i(\calF,\calO_{\calF}(\mathbf a)) = 0 \text{ for } i\neq N.
    \end{align*}
\end{theorem}

For our implementation of Bott's theorem in Macaulay2, see \cite{githubRepo}.

\begin{example}
\label{ex:bott}
    Consider the line bundle $\calO_{\calF}(2,5,0,4,0)$ on $\calF = \calF l(\mathbb C^5)$. We first apply the transposition $(1,2)$:
    \[
    (1,2)\cdot (2,5,0,4,0) = (4,3,0,4,0).
    \]
    Next, we apply $(3,4)$:
    \[
    (3,4)\cdot (4,3,0,4,0) = (4,3,3,1,0).
    \]
    This is non-increasing and we applied two transpositions, so 
Theorem \ref{thm:bottAlg} tells us that
    \begin{align*}
    & H^2(\calF, \calO_{\calF}(2,5,0,4,0)) \neq 0, \text{ and }\\
    & H^i(\calF, \calO_{\calF}(2,5,0,4,0)) = 0 \text{ for } i \neq 2.
    \end{align*}
    Alternatively, we could have computed:
    \[
    \mathbf b = (2,5,0,4,0) - (0,1,2,3,4) = (2,4,-2,1,-4).
    \]
    This has 2 inversions, so Theorem \ref{thm:bottAlternative} gives the same result as above.
\end{example}

\begin{example}
\label{ex:bott2}
    Consider the line bundle $\calO_{\calF}(2,5,4,4,0)$ on $\calF = \calF l(\mathbb C^5)$. We first apply $(1,2)$:
    \[
    \mathbf a' = (1,2)\cdot (2,5,4,4,0) = (4,3,4,4,0).
    \]
    Note that $a'_2 = a'_3-1$, so 
    Theorem \ref{thm:bottAlg} tells us that
    \begin{align*}
    & H^i(\calF, \calO_{\calF}(2,5,4,4,0)) = 0 \text{ for all } i.
    \end{align*}
    Alternatively, consider
    \[
    \mathbf b = (2,5,4,4,0) - (0,1,2,3,4) = (2,4,2,1,-4).
    \]
    Since $b_1 = b_3$, Theorem \ref{thm:bottAlternative} gives the same result.
\end{example}

\subsection{Multigraded regularity}\label{subsec:regularityDefns}

We record some notation and definitions from \cite{MaclaganSmith2004}.
The following definition is a special case; the original definition applies in a much more general setting. In our setting, this definition is equivalent to Definition 6.2 in \cite{MaclaganSmith2004} by their Corollary 6.6.

\begin{definition}\label{defn:regSheaf}
Let $\calG$ be a coherent sheaf on $\mathbb P = \mathbb P^{a_1}\times\cdots\times\mathbb P^{a_k}$ and let $\mathbf d \in \mathbb Z^k$.  We say $\calG$ is \textit{$\mathbf d$-regular} if for all $i>0$ and all $\mathbf c \in \mathbb N^k$ with $\lvert \mathbf c \rvert = i$,
\[
H^i(\mathbb P, \calG(\mathbf d - \mathbf c)) = 0.
\]
The \textit{multigraded regularity} of $\calG$ is the set of all such $\mathbf d$:
\[
\reg(\calG) = \{\mathbf d: \calG \text{ is } \mathbf d\text{-regular}\}.
\]
\end{definition}

We use the notation $|\mathbf c| \coloneqq c_1+\cdots+c_k$ and refer to this quantity as the \textit{total degree} of $\mathbf c$.  Note that $\mathbf d' \geq \mathbf d$ means $\mathbf d_i' \geq \mathbf d_i$ for each $i$.

Although not immediately clear from this version of the definition, the multigraded regularity region is closed under addition of elements in $\mathbb N^k$.  This follows from \cite[Remark 6.3, Corollary 6.6]{MaclaganSmith2004}.

\begin{proposition}\label{prop:closedUnderPos} \cite[Rmk. 6.3, Cor. 6.6]{MaclaganSmith2004}
    If $\mathbf d \in \reg(\calG)$ and $\mathbf d' \geq \mathbf d$, then $\mathbf d' \in \reg(\calG)$.
\end{proposition}

\begin{example}
Consider the flag variety $\calF = \calF l(\mathbb C^5)$, and the coherent sheaf $\calG \coloneqq \iota_* \calO_{\calF}$ on $\mathbb P = \mathbb P\left(\bigwedge \mathbb C^n\right)$.  When asking whether $(-1,5,-4,4)$ is in $\reg(\calG)$, one cohomology group which arises is:
\begin{align*}
H^2(\mathbb P, \iota_*\calO_{\calF}\otimes\calO_{\mathbb P}((-1,5,-4,4)-(2,0,0,0))) &= H^2(\mathbb P, \iota_*\calO_{\calF}\otimes\calO_{\mathbb P}((-3,5,-4,4)))\\
&= H^2(\calF, \calO_{\calF}(2,5,0,4,0)).
\end{align*}
As we saw in Example \ref{ex:bott}, this is nonzero, so $(-1,5,-4,4)$ is not in $\reg(\calG)$.  By Proposition \ref{prop:closedUnderPos}, we see that also $(-1,0,-4,-4)$ is not in $\reg(\calG)$, for example.
\end{example}

Finally, we refine Question \ref{q:intro}:

\begin{question}
    What is the multigraded regularity of $\iota_*\calO_{\calF l(\mathbb C^n)}$?
\end{question}

Maclagan and Smith also introduce a notion of multigraded regularity for modules.  As before, the following definition is a special case of their much more general definition.  In our setting, this definition is equivalent to Definition 4.1 in \cite{MaclaganSmith2004} by their Corollary 4.8.

\begin{definition}\label{defn:regModule}
    Let $S$ be the Cox ring of $\mathbb P = \mathbb P^{a_1}\times\cdots\times\mathbb P^{a_k}$, and $B$ the irrelevant ideal. Let $M$ be a graded $S$-module and let $\mathbf d \in \mathbb Z^k$.  We say $M$ is \textit{$\mathbf d$-regular} if 
    \begin{enumerate}
        \item $H^0_B(M)_{\mathbf d'}=0$ for all $\mathbf d' > \mathbf d$, and
        \item $H^i_B(M)_{\mathbf d -\mathbf c} = 0$ for all $i>0$ and all $\mathbf c \in \mathbf N^k$ with $|\mathbf c| = i-1$. 
        
    \end{enumerate}
    
    The \textit{multigraded regularity} of $M$ is the set of all such $\mathbf d$:
    \[
    \reg(M) = \{\mathbf d: M \text{ is } \mathbf d\text{-regular}\}.
    \]
\end{definition}

If $\mathcal G$ is the coherent sheaf associated to $M$, then in general, $\reg(M) \subseteq \reg(\mathcal G)$.  The precise relationship between $\reg(M)$ and $\reg(\mathcal G)$ is stated in \cite[Proposition 6.4]{MaclaganSmith2004}.  However, we will argue that the two notions coincide in our setting.  

\begin{proposition}\label{prop:sheafVsModule}
    Fix $n \geq 3$.  Let $S$ be the Cox ring of $\mathbb P = \mathbb P \left(\bigwedge \mathbb C^n\right)$ and $B$ the irrelevant ideal.  Let $I$ be the defining ideal of $\calF l(\mathbb C^n)$ in the Pl\"ucker embedding.  Then \[\reg(S/I) = \reg(\iota_*\calO_{\calF l(\mathbb C^n)}).\]
\end{proposition}

\begin{proof}

Let $\mathbf d \in \reg(\iota_*\calO_{\calF l(\mathbb C^n)})$.  By our Corollary \ref{cor:positiveOrthant}, we know $\mathbf d \geq \mathbf 0$.  Since $I$ is prime, we have $H^0_B(S/I) = 0$.  Then by \cite[Proposition 6.4]{MaclaganSmith2004}, it suffices to show
\begin{equation}\label{eqn:naturalMap}
    (S/I)_{\mathbf p} \rightarrow H^0(\mathbb P,\iota_*\calO_{\calF l(\mathbb C^n)}\otimes\calO_{\mathbb P}(\mathbf p))
\end{equation}
is surjective for all $\mathbf p \geq \mathbf d$.

Note that we have
\begin{equation*}
H^0(\mathbb P, \iota_*\calO_{\calF}\otimes \calO_{\mathbb P}(\mathbf p)) = H^0(\calF l(\mathbb C^n),\calO_{\calF l(\mathbb C^n)}(p_1+\cdots p_{n-1},p_2+\cdots+p_{n-1},\ldots,p_{n-1},0),
\end{equation*}
which is an irreducible representation by Bott's theorem.  Furthermore, (\ref{eqn:naturalMap}) is injective by our observation that $H^0_B(S/I) = 0$.  Therefore, either $(S/I)_{\mathbf p} = 0$ or (\ref{eqn:naturalMap}) is an isomorphism.

Suppose $(S/I)_{\mathbf p} = 0$.  Since $S/I$ is generated as an $S$-module by $(S/I)_{\mathbf 0}$, this would imply $(S/I)_{\mathbf q} = 0$ for all $\mathbf q \geq \mathbf p$, a contradiction.  Therefore,   (\ref{eqn:naturalMap}) must be surjective, and we can conclude that $\mathbf d \in \reg(S/I)$.
\end{proof}

 Throughout, we will state results in terms of $\reg(\iota_*\calO_{\calF l(\mathbb C^n)})$.  However, Proposition \ref{prop:sheafVsModule} implies that all the same results apply for $\reg(S/I)$.

\section{Inductive results}\label{sec:inductive}

In this section, we establish two key results relating the regularity of $\calF l(\mathbb C^n)$ to that of $\calF l(\mathbb C^{n-1})$.  Since we will be translating between statements about cohomology on $\calF l(\mathbb C^n)$ and statements about cohomology on $\mathbb P$, it will be convenient to define a function $\varphi: \mathbb Z^{n-1}\rightarrow \mathbb Z^n$ by
\begin{equation}\label{eqn:varphi}
    \varphi(b_1,\ldots,b_{n-1}) = (b_1+\ldots+b_{n-1},b_2+\ldots+b_{n-1},\ldots,b_{n-1},0).
\end{equation}
Note that
\begin{align*}
    \iota_*\calO_{\calF}(\varphi(b_1,\ldots,b_{n-1}))&= \iota_*\calO_{\calF}(b_1+\ldots+b_{n-1},b_2+\ldots+b_{n-1},\ldots,b_{n-1},0)\\
    &= \calO_{\mathbb P}(b_1,b_2,\ldots,b_{n-1}).
\end{align*}

The following statement will simplify our computations:

\begin{lemma}
\label{lemma:reverse}
    Let $\mathbf d \in \mathbb N^{n-1}$ and let $\overline{\mathbf d} = (d_{n-1},\ldots,d_1)$ be given by reversing the entries of $\mathbf d$.  Then $\mathbf d$ is in $\reg(\iota_*\calO_{\calF})$ if and only if $\overline{\mathbf d}$ is in $\reg(\iota_*\calO_{\calF})$.
\end{lemma}

\begin{proof}
    We only need to argue one direction, so suppose $\mathbf d \notin \reg(\iota_*\calO_{\calF})$. Then there is some $i \leq \dim \calF$ and some $\mathbf c \in \mathbb N^{n-1}$ such that $|\mathbf c| = i$ and 
    \begin{equation}\label{eqn:nonzeroCohom}
        H^i(\mathbb P,\iota_*\calO_{\calF}\otimes \calO_{\mathbb P}(\mathbf d-\mathbf c)) \neq 0
    \end{equation}
    
    For ease of notation, write $\mathbf S = \varphi(\mathbf d - \mathbf c)$.  Then (\ref{eqn:nonzeroCohom}) becomes:
    \begin{equation}\label{eqn:nonzeroCohom2}
        H^i(\calF,\calO_{\calF}(\mathbf S)) \neq 0.
    \end{equation}
    By Theorem \ref{thm:bottAlternative}, this means $\mathbf S - (0,1,\ldots,n-1)$ contains exactly $i$ inversions.  
    
    We will argue that
    \begin{equation}\label{eqn:NTSCohom}
        H^i(\mathbb P,\iota_*\calO_{\calF}\otimes \calO_{\mathbb P}(\overline{\mathbf d}-\overline{\mathbf c})) \neq 0
    \end{equation}
    so that $\mathbf{\overline d} \notin \reg(\iota_*\calO_{\calF})$.  We can rewrite (\ref{eqn:NTSCohom}) as:
    \[
    H^i(\calF,\calO_{\calF}(\varphi(\overline{\mathbf d}-\overline{\mathbf c}))) \neq 0.
    \]
    We compute:
    \begin{align*}
        \varphi(\overline{\mathbf d}-\overline{\mathbf c}) &= ((d_1-c_1))+\ldots+(d_{n-1}-c_{n-1}),\ldots, (d_1-c_1)+(d_2-c_2),(d_1-c_1),0)\\
        &= (S_1-S_n,S_1-S_{n-1},\ldots,S_1-S_3,S_1-S_2,0).
    \end{align*}
    Concretely, we need to argue that the following contains $i$ inversions:
    \[
    (S_1-S_n,S_1-S_{n-1},\ldots,S_1-S_3,S_1-S_2,0) - (0,1,\ldots,n-1).
    \]
    The number of inversions is unaffected by adding a multiple of $(1,1,\ldots,1)$, so we can add $(-S_1,\ldots,-S_1)$ to obtain:
    \[
    (-S_n,-S_{n-1},\ldots,-S_1) - (0,1,\ldots,n-1) = (-S_n - 0, -S_{n-1}-1,\ldots, -S_1-(n-1)).
    \]
    The assumption (\ref{eqn:nonzeroCohom2}) tells us concretely that there are $i$ pairs $j<k$ for which $S_j-(j-1)<S_k-(k-1)$.  This inequality holds if and only if $-S_j+j > -S_k + k$, or equivalently $-S_j-(n-j) > -S_k-(n-k)$.  Therefore, $(-S_n - 0, -S_{n-1}-1,\ldots, -S_1-(n-1))$ contains $i$ inversions, proving (\ref{eqn:NTSCohom}).
\end{proof}

Let $\calF'$ denote the the complete flag variety on $\mathbb C^{n-1}$ and $\kappa$ the inclusion $\kappa: \calF' \hookrightarrow \mathbb P\left(\bigwedge\mathbb C^{n-1}\right)= \mathbb P'$.  The following statement tells us that elements outside $\reg(\iota_*\calO_{\calF})$ ``propagate'' as $n$ increases:

\begin{proposition}
\label{prop:inductive}
    Let $\mathbf{d} \in \mathbb Z^{n-1}$.  If $\mathbf{a} = (d_2,\ldots,d_{n-1})$ or $\mathbf{b} = (d_1,\ldots,d_{n-2})$ lies outside of $\reg(\kappa_*\calO_{\calF'})$, then $\mathbf{d}$ lies outside of $\reg(\iota_*\calO_{\calF})$.
\end{proposition}

\begin{proof}
    By Lemma \ref{lemma:reverse}, it suffices to consider one of the two cases.  So suppose $\mathbf{a}$ lies outside of $\reg(\kappa_*\calO_{\calF'})$.  Then there is some $i \leq \dim(\calF')$ and some $\mathbf{c} \in \mathbb N^{n-2}$ such that $|\mathbf{c}| = i$ and 
    \begin{align*}
        H^i(\mathbb P', \kappa_*\calO_{\calF'}\otimes \calO_{\mathbb P'}(\mathbf a - \mathbf c)) \neq 0.
    \end{align*}
    Writing $\mathbf{g} = \varphi'(\mathbf a - \mathbf c)$, this becomes 
    \begin{align*}
        H^i(\calF', \calO_{\calF'}(\mathbf{g}))&\neq 0.
        \end{align*}
    Concretely, this tells us that $\mathbf g - (0,1,\ldots,n-2)$ contains $i$ inversions.  It suffices to show that $\mathbf{d}$ is not in $\reg(\iota_*\calO_{\calF})$ for $d_1 \geq \max(g_j)+i-|\mathbf{a}|$.  Let $\mathbf h = (0,c_1,\ldots,c_{n-2})$.  We compute:
    \begin{align}
        H^i(\mathbb P, \iota_*\calO_{\calF}\otimes \calO_{\mathbb P}(\mathbf d - \mathbf h)) &= H^i(\calF, \calO_{\calF}(\varphi(\mathbf d - \mathbf h)))\nonumber\\
        &=H^i(\calF, \calO_{\calF}(\lvert \mathbf d \rvert - i, g_1,\ldots,g_{n-1})). \label{eqn:inductiveStep}
    \end{align}
    By our assumption on $d_1$, we have $|\mathbf{d}|-i = d_1 + |\mathbf{a}|-i \geq g_j$ for each $j$.  Therefore, 
    \[
    (|\mathbf{d}|-i,g_1,\ldots,g_{n-1}) - (0,1,\ldots,n-1)
    \]
    contains the same number of inversions as  $\mathbf g - (1,2,\ldots,n-1)$.  Since inversions are not affected by adding a multiple of $(1,1,\ldots,1)$, this in turn has the same number of inversions as $\mathbf g - (0,1,\ldots,n-2)$; that is, $i$ inversions.  Therefore, (\ref{eqn:inductiveStep}) is nonzero by Theorem \ref{thm:bottAlternative}. Since also $|\mathbf{h}| = i$, this implies that $\mathbf d$ lies outside $\reg(\iota_*\calO_{\calF})$.
\end{proof}

The following example illustrates Proposition \ref{prop:inductive} and its proof.

\begin{example}
    Let $n = 6$ and consider $\mathbf d = (4,2,0,0,1)$.  We know that $\mathbf a = (2,0,0,1)$ lies outside of $\reg(\kappa_*\calO_{\calF'})$ since
    \begin{align*}
    H^6(\mathbb P', \kappa_*\calO_{\calF'}\otimes\calO_{\mathbb P'}(\mathbf a - (0,2,4,0))) & = H^6(\mathbb P', \kappa_*\calO_{\calF'}\otimes\calO_{\mathbb P'}(2,-2,-4,1))\\
    &= H^6(\calF',\calO_{\calF'}(-3,-5,-3,1,0)).
    \end{align*}
    This is nonzero since $(-3,-5,-3,1,0) - (0,1,2,3,4) = (-3,-6,-5,-2,-4)$ contains $6$ inversions.
    
    In the notation of the above proof, $i = 6$ and $\mathbf c = (0,2,4,0)$.  Let $\mathbf h = (0,0,2,4,0)$, and compute:
    \begin{align*}
        H^6(\mathbb P, \iota_*\calO_{\calF}\otimes\calO_{\mathbb P}(\mathbf d - \mathbf h)) & = H^6(\mathbb P, \iota_*\calO_{\calF}\otimes\calO_{\mathbb P}(4,2,-2,-4,1))\\
        &= H^6(\calF,\calO_{\calF}(1,-3,-5,-3,1,0)).
    \end{align*}
    This is nonzero since $(1,-3,-5,-3,1,0) - (0,1,2,3,4,5) = (1,-4,-7,-6,-3,-5)$ contains $6$ inversions.  Note that assuming $d_1$ is sufficiently large ensures that this sequence contains exactly as many inversions as $(-3,-5,-3,1,0) - (0,1,2,3,4)$.
    It is also worth noting that we could not apply the above proposition to argue that $(2,0,0,1)$ is not in $\reg(\kappa_*\calO_{\calF'})$, since both $(2,0,0)$ and $(0,0,1)$ lie within the regularity region when $n=4$.
\end{example}

Proposition \ref{prop:inductive} allows us to obtain the following outer bound on $\reg(\iota_*\calO_{\calF})$.

\begin{corollary}
\label{cor:positiveOrthant}
    For any $n$, the multigraded regularity of $\iota_*\calO_{\calF}$ is contained in $\mathbb N^{n-1}$
\end{corollary}

\begin{proof}
    By Proposition \ref{prop:inductive}, it suffices to prove that $-1$ lies outside $\reg(\iota_*\calO_{\calF})$ when $n=2$.  To verify this, we compute:
    \begin{align*}
        H^1(\mathbb P^1, \iota_*\calO_{\calF}\otimes\calO_{\mathbb P^1}(-1-1)) &= H^1(\calF, \calO_{\calF}(-2,0)).
    \end{align*}
    This is nonzero since $(-2,0)-(0,1) = (-2,-1)$ contains one inversion.
\end{proof}

We will see in Section \ref{sec:lowValuesN} that $\reg(\iota_*\calO_{\calF}) = \mathbb N^{n-1}$ for $n=3$.  On the other hand, for $n\geq 4$, we will see that the inclusion in Corollary \ref{cor:positiveOrthant} is strict.

There is a partial converse to Proposition \ref{prop:inductive}:

\begin{proposition}
\label{prop:inductive2}
    If $\mathbf d' = (d_1,\ldots,d_{n-2}) \in \reg(\kappa_*\calO_{\calF'})$, then $\mathbf{d} = (d_0,d_1,\ldots,d_{n-2}) \in \reg(\iota_*\calO_{\calF})$ for $d_0$ sufficiently large.  In particular, if $d_0 \geq \dim(\calF)$, then $\mathbf{d} \in \reg(\iota_*\calO_{\calF})$.
\end{proposition}

\begin{proof} 
    The assumption that $\mathbf d' \in \reg(\kappa_*\calO_{\calF'})$ means that for all $i \leq \dim(\calF')$ and all $\mathbf c' \in \mathbb N^{n-2}$ with $|\mathbf{c'}| = i$, we have
    \begin{equation*}
        H^i(\mathbb P', \kappa_*\calO_{\calF'}\otimes\calO_{\mathbb P'}(\mathbf d'-\mathbf c')) = 0.
    \end{equation*}
    Concretely, this means that
    \begin{align*}
        &(d_1+\cdots+d_{n-2}-(c_1+\cdots+c_{n-2}),\ldots,d_{n-2}-d_{n-2},0) - (1,\ldots,n-1)
    \end{align*}
    either contains fewer than $i$ inversions or contains two entries which are equal (or both).  Since $\dim(\calF')$ is the maximum number of inversions for any sequence of this length, we can rephrase this assumption slightly: for any $\mathbf c \in \mathbb N^{n-2}$,
    \begin{equation}\label{eqn:cohomologyF'}
        (d_1+\cdots+d_{n-2}-(c_1+\cdots+c_{n-2}),\ldots,d_{n-2}-d_{n-2},0) - (1,\ldots,n-1)
    \end{equation}
    either contains fewer than $|\mathbf{c}|$ inversions or contains two elements which are equal (or both).

    Now let $i \leq \dim(\calF)$ and let $\mathbf{c} \in \mathbb N^{n-1}$ be such that $|\mathbf{c}| = i$.  We will argue that
\begin{equation*}\label{eqn:cohomologyF}
        H^i(\mathbb P,\iota_*\calO_{\calF}\otimes \calO_{\mathbb P}(\mathbf d - \mathbf c)) = 0.
    \end{equation*}  
    By Borel-Weil-Bott, proving that this cohomology vanishes amounts to computing the number of inversions in the following sequence:
    \begin{equation*}\label{longerTuple}
        \mathbf{a} \coloneqq (d_0+\cdots+d_{n-2}-(c_0+\cdots+c_{n-2}),\ldots,d_{n-2}-c_{n-2},0) - (0,1,\ldots,n-1).
    \end{equation*}
    Note that, writing $\mathbf{a} = (a_0,\ldots,a_{n-1})$, we have $a_j = d_j+\cdots+d_{n-2}-(c_j+\cdots+c_{n-2})-j$ for $j = 0, \ldots, n-2$ and $a_{n-1} = -(n-1)$. We claim that $a_0 > a_j$ for $j>0$.  We have:
    \begin{align*}
        a_0 - a_j &= d_0+\cdots+d_{n-2}-(c_0+\cdots+c_{n-2}) - (d_j+\cdots+d_{n-2}-(c_j+\cdots+c_{n-2})-j)\\
        &= d_0 + \ldots + d_{j-1} - (c_0 + \ldots + c_{j-1}) + j\\
        &> d_0 - (c_0 + \ldots + c_{j-1})\\
        &\geq \dim(\calF) -i\\
        &\geq 0.
    \end{align*}
    So $a_0 > a_j$, as claimed.  Therefore, the number of inversions in $\mathbf{a}$ is equal to that of
    \begin{align*}
        (d_1+\cdots+d_{n-2}-(c_1+\cdots+c_{n-2}),\ldots,d_{n-2}-c_{n-2},0) - (1,\ldots,n-1)
    \end{align*}
    By (\ref{eqn:cohomologyF'}), this contains fewer than $c_1+\cdots+c_{n-2}$ inversions, and thus fewer than $i = c_0+c_1+\cdots+c_{n-2}$ inversions.
\end{proof}

\begin{example}
    For $n = 4$, we can check that $(1,0,0) \in \reg(\kappa_*\calO_{\calF'})$, so this proposition implies $(10,1,0,0) \in \reg(\iota_*\calO_{\calF})$.  However, this gives only an inner bound on the regularity region; we can verify by other methods the stronger statement that $(2,1,0,0) \in \reg(\iota_*\calO_{\calF})$.  We implement this computation via Theorem \ref{thm:bottAlternative} (\cite{githubRepo}).
\end{example}

\section{Standard graded regularity}\label{sec:CMreg}

In this section, we study the image of the complete flag variety in a single projective space.  In particular, consider the map
\begin{equation*}
    j: \calF l(\mathbb C^n) \hookrightarrow \mathbb P^N
\end{equation*}
obtained from composing the Pl\"ucker embedding $\iota: \calF l(\mathbb C^n) \hookrightarrow \mathbb P$ with the Segre embedding $\mathbb P \hookrightarrow \mathbb P^N$.  We can then ask about the standard graded regularity of $j_*\calO_{\calF}$.

\begin{proposition}
The regularity of $j_*\calO_{\calF}$ is $d-1$, where
\begin{equation*}
    d = \dim(\calF) = \frac{n(n-1)}{2}.
\end{equation*}
\end{proposition}

\begin{proof}
    First, we argue that $j_*\calO_{\calF}$ is not $(d-2)$-regular.  In particular, we will show that
    \begin{equation*}
        H^d(\mathbb P^N, j_*\calO_{\calF}((d-2)-d))) \neq 0.
    \end{equation*}
    We compute:
    \begin{align*}
        H^d(\mathbb P^N, j_*\calO_{\calF}((d-2)-d))) &= H^d(\mathbb P, \iota_*\calO_{\calF}(-2,-2,\ldots,-2))\\
        &= H^d(\calF, \calO_{\calF}(-2(n-1),-2(n-2),\ldots,-2,0)).
    \end{align*}
    Theorem \ref{thm:bottAlternative} implies this is nonzero, since
    \begin{align*}
    (-2(n-1),-2(n-2),\ldots,0) - (0,1,\ldots,n-1) &= (-2n+2,-2n+3,\ldots,-2n+n+1)
    \end{align*}
    is strictly increasing, i.e. contains the maximum number of inversions: $\binom{n}{2} = d$.

    Next, we argue that $j_*\calO_{\calF}$ is $(d-1)$-regular.  By definition, we must prove
    \begin{equation*}
        H^i(\mathbb P^N, j_*\calO_{\calF}((d-1)-i))) = 0
    \end{equation*}
    for all $i>0$.  Similarly to the above computation, we have
    \begin{align}
        H^i(\mathbb P^N, j_*\calO_{\calF}((d-1)-i))) &= H^i(\mathbb P, \iota_*\calO_{\calF}(d-1-i,d-1-i,\ldots,d-1-i)) \label{eqn:CMreg}\\
        &= H^i(\calF, \calO_{\calF}((n-1)(d-1-i),(n-2)(d-1-i),\ldots, 0)). \nonumber
    \end{align}
    For $0<i\leq d-1$, we have $d-1-i \geq 0$. This implies $((n-1)(d-1-i),(n-2)(d-1-i),\ldots, 0)$ is non-increasing, so $H^i(\calF, \calO_{\calF}((n-1)(d-1-i),(n-2)(d-1-i),\ldots, 0)) = 0$.  For $i = d$, we have:
    \begin{align*}
        ((n-1)(d-1-i),(n-2)(d-1-i),\ldots, 0) &= (-(n-1),-(n-2),\ldots,0).
    \end{align*}
    Finally, $H^d(\calF, \calO_{\calF}(-(n-1),-(n-2),\ldots,0))$ vanishes by Theorem \ref{thm:bottAlternative} since
    \[
    (-(n-1),-(n-2),\ldots,0) - (0,1,\ldots, n-1) = (-(n-1),-(n-1),\ldots,-(n-1)).
    \]\end{proof}


\section{Low values of $n$}
\label{sec:lowValuesN}

In this section, we fully describe the minimal elements of $\reg(\iota_*\calO_{\calF l(\mathbb C^n)})$ for $n\leq 6$.

\subsection{$n=3$ case}

\begin{proposition}\label{prop:n=3}
    For $n=3$, the multigraded regularity of $\iota_*\calO_{\calF}$ is $\mathbb N^2$.
\end{proposition}

\begin{proof}
    By Proposition \ref{cor:positiveOrthant}, it suffices to show that  $\iota_*\calO_{\calF}$ is $(0,0)$-regular.  This amounts to checking the following conditions: 
    \begin{align*}
        &H^1(\mathbb P,\iota_*\calO_{\calF}\otimes \calO_{\mathbb P}(-1,0)) = 0\\
        &H^1(\mathbb P,\iota_*\calO_{\calF}\otimes \calO_{\mathbb P}(0,-1)) = 0\\
        &H^2(\mathbb P,\iota_*\calO_{\calF}\otimes \calO_{\mathbb P}(-2,0)) = 0\\
        &H^2(\mathbb P,\iota_*\calO_{\calF}\otimes \calO_{\mathbb P}(-1,-1)) = 0\\
        &H^2(\mathbb P,\iota_*\calO_{\calF}\otimes \calO_{\mathbb P}(0,-2)) = 0\\
        &H^3(\mathbb P,\iota_*\calO_{\calF}\otimes \calO_{\mathbb P}(-3,0)) = 0\\
        &H^3(\mathbb P,\iota_*\calO_{\calF}\otimes \calO_{\mathbb P}(-2,-1)) = 0\\
        &H^3(\mathbb P,\iota_*\calO_{\calF}\otimes \calO_{\mathbb P}(-1,-2)) = 0\\
        &H^3(\mathbb P,\iota_*\calO_{\calF}\otimes \calO_{\mathbb P}(0,-3)) = 0.
    \end{align*}
    That is,
    \begin{align}
        &H^1(\calF,\calO_{\calF}(-1,0,0)) = 0 \label{eqn:n3-1}\\
        &H^1(\calF,\calO_{\calF}(-1,-1,0)) = 0 \label{eqn:n3-2}\\
        &H^2(\calF,\calO_{\calF}(-2,0,0)) = 0 \label{eqn:n3-3}\\
        &H^2(\calF,\calO_{\calF}(-2,-1,0)) = 0 \label{eqn:n3-4}\\
        &H^2(\calF,\calO_{\calF}(-2,-2,0)) = 0 \label{eqn:n3-5}\\
        &H^3(\calF,\calO_{\calF}(-3,0,0)) = 0 \label{eqn:n3-6}\\
        &H^3(\calF,\calO_{\calF}(-3,-1,0)) = 0 \label{eqn:n3-7}\\
        &H^3(\calF,\calO_{\calF}(-3,-2,0)) = 0 \label{eqn:n3-8}\\
        &H^3(\calF,\calO_{\calF}(-3,-3,0)) = 0. \label{eqn:n3-9}
    \end{align}
    Statements (\ref{eqn:n3-1}), (\ref{eqn:n3-2}), (\ref{eqn:n3-3}), (\ref{eqn:n3-4}), (\ref{eqn:n3-5}), (\ref{eqn:n3-7}), and (\ref{eqn:n3-8}) follow from  Theorem \ref{thm:bottAlg}; these sheaves in fact have no nonzero cohomology. Statements (\ref{eqn:n3-6}) and (\ref{eqn:n3-9}) also follow from Theorem \ref{thm:bottAlg}, since applying two swaps to $(-3,0,0)$ yields $(-1,-1,-1)$, and applying two swaps to $(-3,-3,0)$ yields $(-2,-2,-2)$.  Therefore, $H^i(\calF,\calO_{\calF}(-3,0,0)) \neq 0$ and $H^i(\calF,\calO_{\calF}(-3,-3,0)) \neq 0$ only for $i=2$.
\end{proof}


\subsection{$n=4$ case}

\begin{proposition}
\label{prop:n4}
    For $n=4$, the minimal elements of $\reg(\iota_*\calO_{\calF})$ are $(1,0,0)$, $(0,1,0)$ and $(0,0,1)$.
\end{proposition}

\begin{proof}
By Proposition \ref{cor:positiveOrthant}, $\reg(\iota_*\calO_{\calF}) \subseteq \mathbb N^3$.  To see that $(0,0,0) \notin \reg(\iota_*\calO_{\calF})$, we compute:
\begin{align*}
    H^4(\mathbb P,\iota_*\calO_{\calF}\otimes\calO_{\mathbb P}(0,-4,0)) &= H^4(\calF, \calO_{\calF}(-4,-4,0,0)).
\end{align*}
This is nonzero by Theorem \ref{thm:bottAlternative} since $(-4,-4,0,0)-(0,1,2,3) = (-4,-5,-2,-3)$ contains four inversions.

It remains to check that $(1,0,0)$, $(0,1,0)$ and $(0,0,1)$ lie in $\reg(\iota_*\calO_{\calF})$.  In fact, by Lemma \ref{lemma:reverse}, it suffices to check only that $(1,0,0)$ and $(0,1,0)$ lie in $\reg(\iota_*\calO_{\calF})$.  For each, there are 83 conditions which must be verified using Theorem \ref{thm:bottAlternative}.  We implement this computation in Macaulay2 (\cite{githubRepo}).
\end{proof}


\subsection{$n=5$ case}

 \begin{proposition} \label{prop:n=5}
    For $n=5$, the minimal elements of $\reg(\iota_*\calO_{\calF})$ are
\begin{itemize}
    \item $(3,0,0,1)$, $(2,0,0,2)$, $(1,0,0,3)$, or
    \item of the form $(a,b,c,d)$ where $a+b+c+d = 3$, other than $(3,0,0,0)$, $(2,0,0,1)$, $(1,0,0,2)$ and $(0,0,0,3)$.
\end{itemize}
\end{proposition}

\begin{proof}

Verifying that the claimed elements lie in the regularity region amounts to a straightforward but tedious application of Theorem \ref{thm:bottAlternative}.  We implement this computation in Macaulay 2; the code can be found at \cite{githubRepo}.

We must also argue that these elements are minimal.  It suffices to argue that $(3,0,0,0)$, $(2,0,0,1)$, $(1,0,0,2)$, $(0,0,0,3)$, as well as $(a,b,c,d)$ with $a+b+c+d = 2$ lie outside $\reg(\iota_*\calO_{\calF})$.  This follows for $(3,0,0,0)$ and $(0,0,0,3)$ from Proposition \ref{prop:inductive} together with Proposition \ref{prop:n4}. To see that $(2,0,0,1)$ also lies outside $\reg(\iota_*\calO_{\calF})$, we compute:
\begin{align*}
        H^6(\mathbb P, \iota_*\calO_{\calF}\otimes \calO_{\mathbb P}((2,0,0,1)-(0,2,4,0))) &= H^6(\mathbb P, \iota_*\calO_{\calF}\otimes \calO_{\mathbb P}(2,-2,-4,1))\\
        &= H^6(\calF, \calO_{\calF}(-3,-5,-3,1,0)).
    \end{align*}
This is nonzero by Theorem \ref{thm:bottAlternative} since $(-3,-5,-3,1,0) - (0,1,2,3,4) = (-3,-6,-5,-2,-4)$ contains 6 inversions.  By Lemma \ref{lemma:reverse}, we see also that $(1,0,0,2)$ lies outside $\reg(\iota_*\calO_{\calF})$.

We also need to show that $(a,b,c,d)$ lies outside $\reg(\iota_*\calO_{\calF})$ when $a+b+c+d = 2$. There are 10 such elements.  By Propositions \ref{prop:inductive} and \ref{prop:n4}, $(2,0,0,0)$ and $(0,0,0,2)$ lie outside $\reg(\iota_*\calO_{\calF})$.  Applying Lemma \ref{lemma:reverse}, there remain only five elements to consider: $(1,1,0,0)$, $(1,0,1,0)$, $(1,0,0,1)$, $(0,2,0,0)$, and $(0,1,1,0)$.  
Appendix \ref{app:n=5} provides, for each of these elements, an example of a failed cohomology condition.

Finally, to see that the listed elements are the only minimal elements, it remains to check that $(a,0,0,0)$ and $(0,0,0,a)$ are not in $\reg(\iota_*\calO_{\calF})$.  This follows from Proposition \ref{prop:inductive} along with Proposition \ref{prop:n4}.
\end{proof}


\subsection{$n=6$ case}

\begin{proposition}\label{prop:n=6}
    For $n=6$, there are 179 minimal elements of $\reg(\iota_*\calO_{\calF})$:
    \begin{itemize}
        \item 147 minimal elements of total degree 6,
        \item 22 minimal elements of total degree 7, and
        \item 10 minimal elements of total degree 8.
    \end{itemize}
\end{proposition}

The minimal elements are listed in Appendix \ref{app:n=6}.

Our proof will rely upon the following lemma, which says  that for elements of total degree at least $8$, it suffices to apply Proposition \ref{prop:inductive}.  As in Section \ref{sec:inductive}, we let $\calF'$ denote the the complete flag variety on $\mathbb C^{5}$ and $\kappa$ the inclusion $\kappa: \calF' \hookrightarrow \mathbb P\left(\bigwedge\mathbb C^{5}\right)= \mathbb P'$.

\begin{lemma}\label{lemma:n=6}
    For $\mathbf a \in \mathbb N^5$, if $\mathbf a \notin \reg(\iota_*\calO_{\calF})$ and $|\mathbf a| \geq 8$, then at least one of $(a_1,\ldots,a_4)$ or $(a_2,\ldots,a_5)$ lies outside $\reg(\kappa_*\calO_{\calF'}).$
\end{lemma}

\begin{proof}
    We argue by induction on $|\mathbf a|$.  For $|\mathbf a| = 8$, see the list of such elements in Appendix \ref{app:n=6}.  By comparing this list to our description of $\reg(\kappa_*\calO_{\calF'}),$ we see that the claim holds for $|\mathbf a| = 8$.

    Next, assume the claim holds for some $N \geq 8$ and all $\mathbf a$ of total degree $N$.  Suppose $\mathbf a$ lies outside $\reg(\iota_*\calO_{\calF})$ and has $|\mathbf a| = N+1$.  If $a_2 =a_3=a_4=0$, then we're done.  If not, observe from Appendix \ref{app:n=6} that $\mathbf a$ lies in the regularity whenever $a_2+a_3+a_4 \geq 6$.  Therefore, we
cannot have $a_1 = 0$ and $a_5=0$.  Suppose, without loss of generality, that $a_1 > 0$.  Since $(a_1-1,a_2,a_3,a_4,a_5)$ lies outside   $\reg(\iota_*\calO_{\calF})$ and has total degree $N$, we  must have that $(a_1-1,a_2,a_3,a_4)$ or $(a_2,a_3,a_4,a_5)$ is not in $\reg(\kappa_*\calO_{\calF'}).$   In the second case, we're done.  In the first case, since we've assumed $a_2,a_3,a_4$ are not all zero, we must have $a_1-1+a_2+a_3+a_4 \leq 3$.  This implies $a_5 \geq N-3$.  Note also that $(a_1,a_2,a_3,a_4,a_5-1)$ must not be in $\reg(\iota_*\calO_{\calF})$, so $(a_1,a_2,a_3,a_4)$ or $(a_2,a_3,a_4,a_5-1)$ must lie outside $\reg(\kappa_*\calO_{\calF'})$.  Since at least one of $a_2,a_3,a_4$ is nonzero and since $a_5-1 \geq N-4 > 3$, we must have $(a_2,a_3,a_4,a_5-1) \in \reg(\kappa_*\calO_{\calF'})$.  Therefore, $(a_1,a_2,a_3,a_4) \notin \reg(\kappa_*\calO_{\calF'})$, finishing the proof of our claim.
\end{proof}

\begin{proof}[Proof of Proposition \ref{prop:n=6}]
As for lower values of $n$, checking that the claimed elements lie in the regularity region is a straightforward application of Theorem \ref{thm:bottAlternative}, but infeasible to complete by hand.  The code for our implementation can be found at \cite{githubRepo}.  

We must also prove that there are no minimal elements of total degree greater than $8$.  Let $\mathbf b \in \reg(\iota_*\calO_{\calF})$ with $|\mathbf b| \geq 9$.  We will argue that $\mathbf b$ is not a minimal element of the regularity.  We must have $b_2+b_3+b_4 \geq 1$.  If $b_2+b_3+b_4 = 1$, then $b_1+b_5 \geq 8$, so without loss of generality $b_1 \geq 4$.  Consider $\mathbf b' \coloneqq (b_1-1,b_2,b_3,b_4,b_5)$.  We have $(b_1-1,b_2,b_3,b_4) \in \reg(\kappa_*\calO_{\calF'})$ since $b_1-1 \geq 3$ and $b_2+b_3+b_4 = 1$.  We also have $(b_2,b_3,b_4,b_5) \in \reg(\kappa_*\calO_{\calF'})$ by Proposition \ref{prop:inductive}.  Then Lemma \ref{lemma:n=6} implies $\mathbf b' \in \reg(\iota_*\calO_{\calF})$.  Since $\mathbf b = \mathbf b' + (1,0,0,0,0)$, this shows that $\mathbf b$ is not minimal.

It remains to consider the case where $b_2+b_3+b_4 \geq 2$.  If $\mathbf b$ were minimal, then Lemma \ref{lemma:n=6} would imply $b_1+b_2+b_3+b_4-1 \leq 3$ or $b_2+b_3+b_4+b_5-1 \leq 3$.  Then we could assume, without loss of generality, that $b_1\geq 5$.  Consider $\mathbf b' \coloneqq (b_1-1, b_2,b_3,b_4,b_5)$.  We have $(b_1-1, b_2,b_3,b_4) \in \reg(\kappa_*\calO_{\calF'})$ since $b_1-1 \geq 4$ and $b_2+b_3+b_4 \geq 2$.  We also have $(b_2,b_3,b_4,b_5) \in \reg(\kappa_*\calO_{\calF'})$ by Proposition \ref{prop:inductive}.  Then Lemma \ref{lemma:n=6} implies $\mathbf b' \in \reg(\iota_*\calO_{\calF})$, so again $\mathbf b$ is not minimal.
\end{proof}

\section{Metrics on multigraded regularity}\label{sec:totalDegree}

The results in Section \ref{sec:lowValuesN} lead us to ask about two different metrics on the regularity regions $\reg(\iota_*\calO_{\calF l(\mathbb C^n)})$, each of which is addressed in a subsection below.

\subsection{Minimal total degree}

Stated informally, we ask the following:

\hspace{25pt} How far from the origin are the minimal elements of the regularity of $\calF l(\mathbb C^n)$?

More precisely, we consider the following function of $n$:
\begin{equation*}\label{eqn:m(n)}
    m(n) \coloneqq \min\{\lvert \mathbf{a} \rvert : \mathbf{a} \in \reg(\iota_*\calO_{\calF l(\mathbb C^n)})\}
\end{equation*}
where $\lvert \mathbf a \rvert = \sum_i a_i$ is the total degree of $\mathbf a$.

Our inductive results allow us to obtain bounds on $m(n)$:

\begin{proposition}\label{prop:m(n)lower}
    For $n \geq 7$, $m(n) \geq \frac{6}{5}n-6$.
\end{proposition}

\begin{proof}
    Suppose $\mathbf a \in \mathbb N^{n-1}$ lies in $\reg(\iota_*\calO_{\calF })$.  We will split $\mathbf a$ into subsequences of length 5 and apply Proposition \ref{prop:inductive}.  Concretely, write $n-1 = 5q+r$ where $0\leq r < 5$.  Consider the following subsequences of $\mathbf a$:
    \begin{equation*}
        (a_1,\ldots,a_5),  (a_6,\ldots,a_{10}),\ldots,( a_{5q-4},\ldots,a_{5q}),(a_{5q+1},\ldots,a_{n-1}).
    \end{equation*}
    For $k = 1,\ldots, q$, we must have $a_{5k-4}+\cdots + a_{5k} \geq 6$ since $m(6) = 6$ by Proposition \ref{prop:n=6}.  Then we have:
    \begin{align*}
        |\mathbf a| &\geq 6q\\
        &= 6\cdot \frac{n-(r+1)}{5}\\
        &\geq \frac{6}{5}(n-5)\\
        &= \frac{6}{5}n-6.
    \end{align*}
\end{proof}

\begin{proposition}\label{prop:m(n)upper}
    For $n \geq 3$, $m(n) \leq \frac{n^3-n-24}{6}$.
\end{proposition}

\begin{proof}
    We know that $(0,0)$ is in $\reg(\iota_*\calO_{\calF l(\mathbb C^3)})$.  By Proposition \ref{prop:inductive2}, $(0,0,6) \in \reg(\iota_*\calO_{\calF l(\mathbb C^4)})$, $(0,0,6,10) \in \reg(\iota_*\calO_{\calF l(\mathbb C^5)})$, and so on.  In general, repeatedly applying Proposition \ref{prop:inductive2} yields:
    \begin{align*}
    m(n) &\leq 6+10+\cdots \frac{n(n-1)}{2}\\
    &= \frac{n^3-n-24}{6}.
    \end{align*}
\end{proof}

So the growth of $m(n)$ is at least linear and at most cubic in $n$.  Examining the data for low values of $n$ suggests the following:

\begin{conj}\label{conj:m(n)}
    For $n \geq 4$, $m(n) = \binom{n-2}{2}$.
\end{conj}

\subsection{Multiples of $(1,1,\ldots,1)$}

The second metric which we study is given by intersecting the regularity region with the line spanned by $(1,1,\ldots,1)$.

\begin{question}\label{QAllOnes}
    For which values of $a$ is $(a,a,\ldots,a)$ in $\reg(\iota_*\calO_{\calF})$?
\end{question}

Note that by definition, $\iota_*\calO_{\calF}$ is $(a,a,\ldots,a)$-regular if for all $i>0$ and all $|\mathbf c| = i$,
\begin{equation*}
    H^i(\mathbb P,\iota_*\calO_{\calF}((a,\ldots,a)-\mathbf c)) =0.
\end{equation*}

We establish a lower bound on the answer to Question \ref{QAllOnes}.

\begin{proposition}
\label{prop:dLowerBound}
If $(a,a,\ldots,a)$ is in $\reg(\iota_*\calO_{\calF})$, then $a \geq \frac{n-3}{2}$.
\end{proposition}

\begin{proof}
    For $n = 2k$ and $a = \lfloor\frac{n-3}{2}\rfloor  = k-2$, we will argue that $(a,a,\ldots,a)$ is not in $\reg(\iota_*\calO_{\calF})$.  Let $\mathbf{b} = (k,k,\ldots,k)$, and note that $\lvert\mathbf{b}\rvert = k\cdot (n-1) = \frac{n(n-1)}{2}$.  Writing $i = \frac{n(n-1)}{2}$, we have
    \begin{align*}
        H^i(\mathbb P,&\iota_*\calO_{\calF}\otimes \calO_{\mathbb P}((a,a,\ldots,a)-\mathbf b)) \\
        &= H^i(\calF,\calO_{\calF}((n-1)(k-2),\ldots,k-2,0)\otimes\calO_{\calF}(-(n-1)k,\ldots,-k,0))\\
        &= H^i(\calF,\calO_{\calF}(-2(n-1),\ldots,-2,0)).
    \end{align*}
    This is nonzero since
    \begin{align*}
        &(-2(n-1),-2(n-2),\ldots,-2,0) - (0,1,\ldots,-(n-2),-(n-1))\\
        &= (-2(n-1),-2(n-2)-1,\ldots,-2-(n-2),-(n-1))
    \end{align*}
    is strictly increasing, i.e. contains the maximum number of inversions: $\binom{n}{2} = i$.  Therefore, $(a,a,\ldots,a) \notin \reg(\iota_*\calO_{\calF})$ for $a \leq k-2 = \lfloor\frac{n-3}{2}\rfloor$.

    For $n = 2k+1$, we apply Proposition \ref{prop:inductive} together with the above argument to conclude that $(a,a,\ldots,a) \notin \reg(\iota_*\calO_{\calF})$ for $a \leq k-2 = \frac{n-3}{2} - 1$.  That is, $(a,a,\ldots,a) \notin \reg(\iota_*\calO_{\calF})$ for $a < \frac{n-3}{2}$.
\end{proof}

We conjecture the following answer to Question \ref{QAllOnes}.

\begin{conj}
    $(a,a,\ldots,a)$ is in $\reg(\iota_*\calO_{\calF})$ if and only if $a \geq \frac{n-3}{2}$.
\end{conj}

\bibliographystyle{alpha}
\bibliography{bibliography}

\appendix

\section{Computations for low values of $n$}

\subsection{n = 5}\label{app:n=5}

For each $\mathbf a \in \mathbb N^4$ with $|\mathbf a| = 2$, the table below provides values of $i$ and $\mathbf c$ such that $|\mathbf c| = i$ and 
\[
H^i(\mathbb P, \iota_*\calO_{\calF}\otimes \calO_{\mathbb P}(\mathbf a-\mathbf c)) \neq 0.
\]
For each $\mathbf a$, there are several such $i$ and $\mathbf c$.  One can check that those in the list are minimal.  See code at \cite{githubRepo}.

\begin{center}
\begin{tabular}{ |c|c|c|c| } 
 \hline
$\mathbf a$ & i  & $\mathbf c$ \\
  \hline
  $(2,0,0,0)$ & $4$ & $(0,0,4,0)$\\ 
  \hline
$(1,1,0,0)$ & $7$ & $(0,3,4,0)$\\ 
  \hline
$(1,0,1,0)$ & $7$ & $(0,2,5,0)$\\ 
  \hline
$(1,0,0,1)$ & $5$ & $(0,5,0,0)$\\ 
  \hline
$(0,2,0,0)$ & $8$ & $(0,5,3,0)$\\ 
  \hline
$(0,1,1,0)$ & $8$ & $(0,4,4,0)$\\ 
  \hline
$(0,1,0,1)$ & $7$ & $(0,5,2,0)$\\ 
  \hline
$(0,0,2,0)$ & $8$ & $(0,3,5,0)$\\ 
  \hline
$(0,0,1,1)$ & $7$ & $(0,4,3,0)$\\ 
  \hline
  $(0,0,0,2)$ & $4$ & $(0,4,0,0)$\\ 
  \hline
\end{tabular}
\end{center}

\newpage

\subsection{n=6}\label{app:n=6}

We include a list of the minimal elements of $\reg(\iota_*\calO_{\calF})$.  To verify, see code at  \cite{githubRepo}.

\begin{itemize}
    \item Total degree 6:
    {\tiny \begin{align*} 
    & \left(3,\,2,\,1,\,0,\,0\right),\,\left(2,\,3,\,1,\,0,\,0\right),\,\left(1,\,4,\,1,\,0,\,0\right),\,\left(0,\,5,\,1,\,0,\,0\right),\,\left(3,\,1,\,2,\,0,\,0\right), \\
    & \left(2,\,2,\,2,\,0,\,0\right),\,\left(1,\,3,\,2,\,0,\,0\right),\,\left(0,\,4,\,2,\,0,\,0\right),\,\left(3,\,0,\,3,\,0,\,0\right),\,\left(2,\,1,\,3,\,0,\,0\right), \\
    &\left(1,\,2,\,3,\,0,\,0\right),\,\left(0,\,3,\,3,\,0,\,0\right),\,\left(2,\,0,\,4,\,0,\,0\right),\,\left(1,\,1,\,4,\,0,\,0\right),\,\left(0,\,2,\,4,\,0,\,0\right), \\
    &\left(1,\,0,\,5,\,0,\,0\right),\,\left(0,\,1,\,5,\,0,\,0\right),\,\left(0,\,0,\,6,\,0,\,0\right),\,\left(3,\,2,\,0,\,1,\,0\right),\,\left(2,\,3,\,0,\,1,\,0\right),\\
    &\left(1,\,4,\,0,\,1,\,0\right),\,\left(0,\,5,\,0,\,1,\,0\right),\,\left(3,\,1,\,1,\,1,\,0\right),\,\left(2,\,2,\,1,\,1,\,0\right),\,\left(1,\,3,\,1,\,1,\,0\right),\\
    &\left(0,\,4,\,1,\,1,\,0\right),\,\left(3,\,0,\,2,\,1,\,0\right),\,\left(2,\,1,\,2,\,1,\,0\right),\,\left(1,\,2,\,2,\,1,\,0\right),\,\left(0,\,3,\,2,\,1,\,0\right),\\
    &\left(2,\,0,\,3,\,1,\,0\right),\,\left(1,\,1,\,3,\,1,\,0\right),\,\left(0,\,2,\,3,\,1,\,0\right),\,\left(1,\,0,\,4,\,1,\,0\right),\,\left(0,\,1,\,4,\,1,\,0\right),\\
    &\left(0,\,0,\,5,\,1,\,0\right),\,\left(3,\,1,\,0,\,2,\,0\right),\,\left(2,\,2,\,0,\,2,\,0\right),\,\left(1,\,3,\,0,\,2,\,0\right),\,\left(0,\,4,\,0,\,2,\,0\right),\\
    &\left(3,\,0,\,1,\,2,\,0\right),\,\left(2,\,1,\,1,\,2,\,0\right),\,\left(1,\,2,\,1,\,2,\,0\right),\,\left(0,\,3,\,1,\,2,\,0\right),\,\left(2,\,0,\,2,\,2,\,0\right),\\
    &\left(1,\,1,\,2,\,2,\,0\right),\,\left(0,\,2,\,2,\,2,\,0\right),\,\left(1,\,0,\,3,\,2,\,0\right),\,\left(0,\,1,\,3,\,2,\,0\right),\,\left(0,\,0,\,4,\,2,\,0\right),\\
    &\left(3,\,0,\,0,\,3,\,0\right),\,\left(2,\,1,\,0,\,3,\,0\right),\,\left(1,\,2,\,0,\,3,\,0\right),\,\left(0,\,3,\,0,\,3,\,0\right),\,\left(2,\,0,\,1,\,3,\,0\right),\\
    &\left(1,\,1,\,1,\,3,\,0\right),\,\left(0,\,2,\,1,\,3,\,0\right),\,\left(1,\,0,\,2,\,3,\,0\right),\,\left(0,\,1,\,2,\,3,\,0\right),\,\left(0,\,0,\,3,\,3,\,0\right),\\
    &\left(2,\,0,\,0,\,4,\,0\right),\,\left(1,\,1,\,0,\,4,\,0\right),\,\left(0,\,2,\,0,\,4,\,0\right),\,\left(1,\,0,\,1,\,4,\,0\right),\,\left(0,\,1,\,1,\,4,\,0\right),\\
    &\left(0,\,0,\,2,\,4,\,0\right),\,\left(1,\,0,\,0,\,5,\,0\right),\,\left(0,\,1,\,0,\,5,\,0\right),\,\left(0,\,0,\,1,\,5,\,0\right),\,\left(2,\,3,\,0,\,0,\,1\right),\\
    &\left(1,\,4,\,0,\,0,\,1\right),\,\left(0,\,5,\,0,\,0,\,1\right),\,\left(2,\,2,\,1,\,0,\,1\right),\,\left(1,\,3,\,1,\,0,\,1\right),\,\left(0,\,4,\,1,\,0,\,1\right),\\
    &\left(2,\,1,\,2,\,0,\,1\right),\,\left(1,\,2,\,2,\,0,\,1\right),\,\left(0,\,3,\,2,\,0,\,1\right),\,\left(2,\,0,\,3,\,0,\,1\right),\,\left(1,\,1,\,3,\,0,\,1\right),\\
    &\left(0,\,2,\,3,\,0,\,1\right),\,\left(1,\,0,\,4,\,0,\,1\right),\,\left(0,\,1,\,4,\,0,\,1\right),\,\left(0,\,0,\,5,\,0,\,1\right),\,\left(2,\,2,\,0,\,1,\,1\right),\\
    &\left(1,\,3,\,0,\,1,\,1\right),\,\left(0,\,4,\,0,\,1,\,1\right),\,\left(2,\,1,\,1,\,1,\,1\right),\,\left(1,\,2,\,1,\,1,\,1\right),\,\left(0,\,3,\,1,\,1,\,1\right),\\
    &\left(2,\,0,\,2,\,1,\,1\right),\,\left(1,\,1,\,2,\,1,\,1\right),\,\left(0,\,2,\,2,\,1,\,1\right),\,\left(1,\,0,\,3,\,1,\,1\right),\,\left(0,\,1,\,3,\,1,\,1\right),\\
    &\left(0,\,0,\,4,\,1,\,1\right),\,\left(2,\,1,\,0,\,2,\,1\right),\,\left(1,\,2,\,0,\,2,\,1\right),\,\left(0,\,3,\,0,\,2,\,1\right),\,\left(2,\,0,\,1,\,2,\,1\right),\\
    &\left(1,\,1,\,1,\,2,\,1\right),\,\left(0,\,2,\,1,\,2,\,1\right),\,\left(1,\,0,\,2,\,2,\,1\right),\,\left(0,\,1,\,2,\,2,\,1\right),\,\left(0,\,0,\,3,\,2,\,1\right),\\
    &\left(2,\,0,\,0,\,3,\,1\right),\,\left(1,\,1,\,0,\,3,\,1\right),\,\left(0,\,2,\,0,\,3,\,1\right),\,\left(1,\,0,\,1,\,3,\,1\right),\,\left(0,\,1,\,1,\,3,\,1\right),\\
    &\left(0,\,0,\,2,\,3,\,1\right),\,\left(1,\,0,\,0,\,4,\,1\right),\,\left(0,\,1,\,0,\,4,\,1\right),\,\left(0,\,0,\,1,\,4,\,1\right),\,\left(1,\,3,\,0,\,0,\,2\right),\\
    &\left(0,\,4,\,0,\,0,\,2\right),\,\left(1,\,2,\,1,\,0,\,2\right),\,\left(0,\,3,\,1,\,0,\,2\right),\,\left(1,\,1,\,2,\,0,\,2\right),\,\left(0,\,2,\,2,\,0,\,2\right),\\
    &\left(1,\,0,\,3,\,0,\,2\right),\,\left(0,\,1,\,3,\,0,\,2\right),\,\left(0,\,0,\,4,\,0,\,2\right),\,\left(1,\,2,\,0,\,1,\,2\right),\,\left(0,\,3,\,0,\,1,\,2\right),\\
    &\left(1,\,1,\,1,\,1,\,2\right),\,\left(0,\,2,\,1,\,1,\,2\right),\,\left(1,\,0,\,2,\,1,\,2\right),\,\left(0,\,1,\,2,\,1,\,2\right),\,\left(0,\,0,\,3,\,1,\,2\right),\\
    &\left(1,\,1,\,0,\,2,\,2\right),\,\left(0,\,2,\,0,\,2,\,2\right),\,\left(1,\,0,\,1,\,2,\,2\right),\,\left(0,\,1,\,1,\,2,\,2\right),\,\left(0,\,0,\,2,\,2,\,2\right),\\
    &\left(1,\,0,\,0,\,3,\,2\right),\,\left(0,\,1,\,0,\,3,\,2\right),\,\left(0,\,0,\,1,\,3,\,2\right),\,\left(0,\,3,\,0,\,0,\,3\right),\,\left(0,\,2,\,1,\,0,\,3\right),\\
    &\left(0,\,1,\,2,\,0,\,3\right),\,\left(0,\,0,\,3,\,0,\,3\right),\,\left(0,\,2,\,0,\,1,\,3\right),\,\left(0,\,1,\,1,\,1,\,3\right),\,\left(0,\,0,\,2,\,1,\,3\right),\\
    &\left(0,\,1,\,0,\,2,\,3\right),\,\left(0,\,0,\,1,\,2,\,3\right)
    \end{align*}}%
    \item Total degree 7:
    {\tiny\begin{align*}
    &\left(4,\,1,\,1,\,0,\,1\right),\,\left(4,\,0,\,2,\,0,\,1\right),\,\left(4,\,1,\,0,\,1,\,1\right),\,\left(4,\,0,\,1,\,1,\,1\right),\,\left(4,\,0,\,0,\,2,\,1\right),\\
    &\left(3,\,2,\,0,\,0,\,2\right),\,\left(3,\,1,\,1,\,0,\,2\right),\,\left(3,\,0,\,2,\,0,\,2\right),\,\left(3,\,1,\,0,\,1,\,2\right),\,\left(3,\,0,\,1,\,1,\,2\right),\\
    &\left(3,\,0,\,0,\,2,\,2\right),\,\left(2,\,2,\,0,\,0,\,3\right),\,\left(2,\,1,\,1,\,0,\,3\right),\,\left(2,\,0,\,2,\,0,\,3\right),\,\left(2,\,1,\,0,\,1,\,3\right),\\
    &\left(2,\,0,\,1,\,1,\,3\right),\,\left(2,\,0,\,0,\,2,\,3\right),\,\left(1,\,2,\,0,\,0,\,4\right),\,\left(1,\,1,\,1,\,0,\,4\right),\,\left(1,\,0,\,2,\,0,\,4\right),\\
    &\left(1,\,1,\,0,\,1,\,4\right),\,\left(1,\,0,\,1,\,1,\,4\right)
    \end{align*}}
    \item Total degree 8: 
    {\tiny\begin{align*}
    &\left(5,\,0,\,1,\,0,\,2\right),\,\left(5,\,0,\,0,\,1,\,2\right),\,\left(4,\,1,\,0,\,0,\,3\right),\,\left(4,\,0,\,1,\,0,\,3\right),\,\left(4,\,0,\,0,\,1,\,3\right),\\
    &\left(3,\,1,\,0,\,0,\,4\right),\,\left(3,\,0,\,1,\,0,\,4\right),\,\left(3,\,0,\,0,\,1,\,4\right),\,\left(2,\,1,\,0,\,0,\,5\right),\,\left(2,\,0,\,1,\,0,\,5\right)
    \end{align*}}
\end{itemize}

\newpage
We also include a list of elements of total degree 8 which lie outside $\reg(\iota_*\calO_{\calF})$.

\begin{align*}
    &\left(8,\,0,\,0,\,0,\,0\right),\,\left(7,\,1,\,0,\,0,\,0\right),\,\left(6,\,2,\,0,\,0,\,0\right),\,\left(5,\,3,\,0,\,0,\,0\right),\,\left(4,\,4,\,0,\,0,\,0\right),\\
&\left(3,\,5,\,0,\,0,\,0\right),\,\left(2,\,6,\,0,\,0,\,0\right),\,\left(1,\,7,\,0,\,0,\,0\right),\,\left(0,\,8,\,0,\,0,\,0\right),\,\left(7,\,0,\,1,\,0,\,0\right),\\
&\left(6,\,1,\,1,\,0,\,0\right),\,\left(6,\,0,\,2,\,0,\,0\right),\,\left(7,\,0,\,0,\,1,\,0\right),\,\left(6,\,1,\,0,\,1,\,0\right),\,\left(6,\,0,\,1,\,1,\,0\right),\\
&\left(6,\,0,\,0,\,2,\,0\right),\,\left(0,\,0,\,0,\,8,\,0\right),\,\left(7,\,0,\,0,\,0,\,1\right),\,\left(6,\,1,\,0,\,0,\,1\right),\,\left(5,\,2,\,0,\,0,\,1\right),\\
&\left(6,\,0,\,1,\,0,\,1\right),\,\left(6,\,0,\,0,\,1,\,1\right),\,\left(0,\,0,\,0,\,7,\,1\right),\,\left(6,\,0,\,0,\,0,\,2\right),\,\left(5,\,1,\,0,\,0,\,2\right),\\
&\left(0,\,0,\,0,\,6,\,2\right),\,\left(5,\,0,\,0,\,0,\,3\right),\,\left(0,\,0,\,0,\,5,\,3\right),\,\left(4,\,0,\,0,\,0,\,4\right),\,\left(0,\,0,\,0,\,4,\,4\right),\\
&\left(3,\,0,\,0,\,0,\,5\right),\,\left(2,\,0,\,0,\,1,\,5\right),\,\left(1,\,0,\,0,\,2,\,5\right),\,\left(0,\,0,\,0,\,3,\,5\right),\,\left(2,\,0,\,0,\,0,\,6\right),\\
&\left(1,\,1,\,0,\,0,\,6\right),\,\left(0,\,2,\,0,\,0,\,6\right),\,\left(1,\,0,\,1,\,0,\,6\right),\,\left(0,\,1,\,1,\,0,\,6\right),\,\left(0,\,0,\,2,\,0,\,6\right),\\
&\left(1,\,0,\,0,\,1,\,6\right),\,\left(0,\,1,\,0,\,1,\,6\right),\,\left(0,\,0,\,1,\,1,\,6\right),\,\left(0,\,0,\,0,\,2,\,6\right),\,\left(1,\,0,\,0,\,0,\,7\right),\\
&\left(0,\,1,\,0,\,0,\,7\right),\,\left(0,\,0,\,1,\,0,\,7\right),\,\left(0,\,0,\,0,\,1,\,7\right),\,\left(0,\,0,\,0,\,0,\,8\right).
\end{align*}

\end{document}

%% file: preamble.tex
\usepackage{times, amsthm, amssymb, amsmath, amsfonts, graphicx, mathrsfs}
\usepackage{mathtools}
\usepackage{tikz} 
\usetikzlibrary{arrows, cd, matrix, positioning}
\usepackage{xcolor}
\usepackage{bbm}
\usepackage{xypic}
\usepackage{mathscinet}
\usepackage{hyperref}
\usepackage{wrapfig}
\usepackage{subcaption}

\newtheorem{theorem}{Theorem}[section]
\newtheorem{lemma}[theorem]{Lemma}
\newtheorem{proposition}[theorem]{Proposition}
\newtheorem{corollary}[theorem]{Corollary}
\newtheorem{question}[theorem]{Question}

\newtheorem{introtheorem}{Theorem}

\theoremstyle{definition}
\newtheorem{definition}[theorem]{Definition}
\newtheorem{example}[theorem]{Example}
\newtheorem{conj}[theorem]{Conjecture}
 
\theoremstyle{remark}

\newcommand{\reg}{\operatorname{reg}}

\newcommand{\calF}{{\mathcal{F}}}
\newcommand{\calG}{{\mathcal{G}}}

\newcommand{\calO}{{\mathcal{O}}}
